\documentclass[11pt]{amsart}
\usepackage{lmodern}
\usepackage{amsmath}
\usepackage[T1]{fontenc}
\usepackage{amssymb,amscd,latexsym,epsfig,hyperref}
\usepackage[mathscr]{euscript}
\usepackage{graphicx}
\usepackage{tikz-cd}
\usepackage{amsthm}
\usepackage{xcolor}

\DeclareFontFamily{OT1}{rsfs}{}
\DeclareFontShape{OT1}{rsfs}{n}{it}{<-> rsfs10}{}
\DeclareMathAlphabet{\curly}{OT1}{rsfs}{n}{it}

\makeatletter
\newcommand{\eqnum}{\refstepcounter{equation}\textup{\tagform@{\theequation}}}
\makeatother

\makeatletter \@addtoreset{equation}{section} \makeatother

\renewcommand{\theequation}{\thesection.\arabic{equation}}
\numberwithin{equation}{subsection}

\theoremstyle{definition}
\newtheorem{dfn}[equation]{Definition}
\newtheorem{lem}[equation]{Lemma}
\newtheorem{cor}[equation]{Corollary}
\newtheorem{prop}[equation]{Proposition}
\newtheorem{rmk}[equation]{Remark}

\newtheorem*{prop*}{Proposition}

\newcommand\rurl[1]{
	\href{http://#1}{\nolinkurl{#1}}
}

\newtheorem{claim}[equation]{Claim}

\newcommand{\E}{\mathcal{E}}
\newcommand{\coker}{\mathrm{coker}}
\newcommand{\EE}{\curly{E}}
\newcommand{\strE}{\mathsf{E}}
\newcommand{\tr}{\mathrm{tr}}
\newcommand{\rk}{\mathrm{rank}}

\newcommand{\Spec}{\textbf{Spec}}

\newcommand{\id}{\mathrm{id}}
\newcommand{\C}{\textbf{C}}
\newcommand{\R}{\textbf{R}}

\newcommand{\OO}{\curly{O}}

\newcommand{\Endd}{\curly{E}\!\mathit{nd}}
\newcommand{\Hom}{\mathrm{Hom}}
\newcommand{\Homm}{\curly{H}\! \mathit{om}}
\newcommand{\RHomm}{\mathbf{R} \curly{H}om}
\newcommand{\Extt}{\curly{E} \! \mathit{xt}}
\newcommand{\Ext}{\mathrm{Ext}}

\newcommand{\ev}{\mathrm{ev}}

\newcommand{\ch}{\mathrm{ch}}
\newcommand{\F}{\mathcal{F}}
\newcommand{\At}{\mathrm{At}}

\newcommand{\N}{\curly{N}}

\newcommand{\so}{\mathfrak{so}}

\newcommand{\spp}{\mathfrak{sp}}

\begin{document}
\title{A virtual structure for symplectic Higgs bundles}
\author{Simon Schirren\\}
	\maketitle
    
	\tableofcontents
    \textbf{Abstract.} We define a perfect obstruction theory for a moduli of symplectic Higgs sheaves $(E,\phi)$ on projective surfaces $S$. Key to this is a minimality assumption on $\ch(E)$ that forces all $E$ to be locally free. This might have implications to define a virtual count and $Sp(r)$-Vafa-Witten invariants.
\section{Introduction}
In \cite{VW}, Vafa and Witten discuss solutions to the supersymmetric Yang-Mills equations and explain relations to the Euler characteristic of the instanton moduli space of sheaves on a real 4-manifold, underlying a smooth complex surface $S$. A mathematical definition for such a moduli space of \textit{Higgs sheaves} on $S$ and an enumerative geometry yielding \textit{Vafa-Witten invariants} has been proposed by Tanaka and Thomas \cite{TT1} for the gauge groups $U(r)$ and $SU(r)$.\\
Fixing invariants $\ch(E)$, Higgs sheaves $(E,\phi)$ - where $E$ is a torsion-free sheaf on $S$ of $\rk(E)=r$ and $\phi \in \Hom(E,E\otimes K_S)$ - appear in the stable locus $\N\subset T^*[-1]\curly{M}$ inside the $(-1)$-shifted cotangent bundle to the moduli stack of torsion-free sheaves $\curly{M}$.\\
Via spectral theory, $\N$ can equivalently be described as a moduli space for stable torsion sheaves $\E_\phi$ on $X=\mathrm{Tot}(K_S)\xrightarrow{\pi} S$, a non-compact Calabi-Yau threefold. $\N$ admits a \textit{virtual cotangent bundle} $V^{\bullet}=[V^{-1} \rightarrow V^0]$ that admits a $(-1)$-\textit{shifted symplectic structure} $$\theta: V^\bullet \rightarrow V^{\bullet,\vee}[1].$$
The corresponding Donaldson-Thomas type invariant comes from the associated symmetric $U(r)$-perfect obstruction theory on $\N$, $$\psi: V^\bullet \rightarrow \mathbf{L}_\N.$$
In order to define a \textit{virtual count}, one uses its natural $\C^\times$-action $$(E,\phi) \mapsto (E,\lambda\phi)$$ and does equivariant localization to define the $U(r)$-Vafa-Witten invariant 
as in \cite[3.18]{TT1}.\\
Unfortunately, this invariant turns out to be zero in many cases, so one has to modify the moduli space, which is roughly done as follows: Taking determinants $$ E \mapsto \det(E),$$
we get a map
\begin{equation}\label{shifteddiff}
\begin{tikzcd}[column sep=2cm] 
\N \arrow[r] \arrow{r}{(\det\pi_*,\tr)}& \mathbf{Pic}(S) \times \Gamma(K_S) 
\end{tikzcd}
\end{equation}
given by $$(E,\phi) \rightarrow (\det(E),\tr(\phi)).$$
As $\mathbf{Pic}(S)$ has trivial tangent bundle and $H^1(\OO_S)^*\cong H^1(K_S)$, one may think of \ref{shifteddiff} as an instance of a map
\begin{equation*}
\begin{tikzcd}[column sep=1cm] 
\mathbf{T}^*[-1]\curly{M} \arrow{r}& T^*[-1]\mathbf{Pic}(S),
\end{tikzcd}
\end{equation*}
which, unfortunately, we have not yet succeeded in defining properly.\\
The sub-moduli space $\N_{SU(r)}\subset \N$ relative $\mathbf{Pic}(S)\times\Gamma(K_S)$ admits again a symmetric obstruction theory and defines a better invariant via $\C^\times$-localization, which gives access to more surfaces $S$, see \cite[6]{TT1}. This invariant is given by
$$\mathsf{VW}_{\N_{SU(r)}}=\int_{[\N_{SU(r)}^{\C^\times}]^{vir}} \frac{1}{e(\nu^{vir})} \in \mathbf{Q}$$
Here, the denominator is the Euler class of the virtual normal bundle $\nu^{vir}$ to the fixed locus $\N^{\C^\times} \subset \N$.\\
In other words, this is an enumerative theory for Higgs bundles of gauge group $SL(r)$, i.e. those pairs $(E,\phi)$ of $\N$, where $\det(E)\cong \mathcal{L}$ is fixed and $\tr(\phi)=0$. This has been computed by many authors, see e.g. the work of \cite{GK} and \cite{GKL} and for Enrique surfaces, recent work of \cite{O}. \\
In view of $S$-duality, we interested in defining similar invariants for Higgs sheaves of other gauge groups, such as the orthosymplectic ones $O(r)$ and $Sp(r)$. 
\subsection{Outline of the paper}
The idea is that bundles $E$ of $\curly{M}$ of gauge group $O(r)$ or $Sp(r)$ can be expressed as fixed points of an involution $$E \mapsto E^*.$$
The involution gives an action on the deformation-obstruction complex at each point, $$\mathbf{R}\Gamma(\Homm(E,E)[1])\rightarrow \mathbf{R}\Gamma(\Homm(E^*,E^*)[1]),$$
given by $$\phi \mapsto -\phi^*.$$ 
Over $\N\subset T^*[-1]\curly{M}$, we get an action on stable sheaves given by $$\iota: (E,\phi)\mapsto (E^*,-\phi^*).$$
Fixed pairs of $\N^\iota \subset \N$ under this action are given by triples $(E,\phi,f)$, with a non-degenerate pairing $$f:E\xrightarrow{\sim} E^*$$ intertwining $\phi$ with $-\phi^*$, i.e. there is a commutative square
\begin{center}
        \begin{tikzcd}
            E \arrow[r,"f"]\arrow[d,"\phi"] &  E^* \arrow[d,"-\phi^*"] \\
            E\otimes K_S \arrow[r,"f\otimes1"] & E^*\otimes K_S.
        \end{tikzcd}
    \end{center}
It turns out (using that stable sheaves are simple) that $f$ is \textit{either} symmetric $f^*=f$ and $E$ is an $O(r)$-bundle together with a section $\phi$ of $\so(E)\otimes K_S$ \textit{or} $f^*=-f$ is symplectic and $E$ is an $Sp(r)$-bundle together with a section $\phi$ of $\spp(E)\otimes K_S$. Accordingly, in Sec. \ref{fixedlocus} we will show that we get under certain assumptions two components $\N_{O(r)}$ and $\N_{Sp(r)}$ of $\N^\iota$.\\
To address the virtual structure, we need to phrase this action in terms of spectral sheaves $\E_\phi$ on $X$, which is done in Sec. \ref{spectral}. The deformation-obstruction theory of spectral sheaves is \textit{symmetric}
$$\Ext^1(\E_\phi,\E_\phi)\cong \Ext^2(\E_\phi,\E_\phi)^*,$$
and we want to lift $\iota$ to an action on the complex $$\iota_*: \R\Gamma(\Homm(\E_{\phi},\E_{\phi})[1]) \rightarrow \R\Gamma(\Homm(\E_{-\phi^*},\E_{-\phi^*})[1]).$$
This will be done directly in families of spectral families $\EE$, where will get an action of $\iota$ on $$\mathbf{T}_\N^{vir}=\tau^{[0,1]}\RHomm_{p_X}(\EE,\EE)[1],$$ compatible with the natural action on $\mathbf{T}_\N=\mathbf{L}_\N^\vee$ via the Atiyah class $\At_{\EE,\N}$, see Sec. \ref{virtualdiff}.\\
\subsection{Main result}
This is enough to meet the requirements of equivariant localization in the sense of \cite{GP}, where we replace $\C^\times$ by $\langle \iota \rangle$. The final part is analogous to \cite{Sch}, where we use the $U(r)$-obstruction theory of \cite{TT1}, which is now $\iota$-equivariant and localize it to $\N_{Sp(r)}\subset \N^\iota$. This allows us to set up a perfect obstruction theory on $\N_{Sp(r)}$, which endows it with a virtual cycle $[\N_{Sp(r)}]^{vir}$, see Prop. \ref{finalthm}.\\
Key to all of this is an assumption on $c_2(E)$ when fixing $\ch(E)$ that assures that all pairs of $\N$ are in fact locally free. Only this will make $(E,\phi) \mapsto (E^*,-\phi^*)$ into a globally well-defined involution, see \ref{moduli}.
\subsection{Outlook}
We are interested in generalizations to arbitrary $\ch(E)$ that would include torsion-free sheaves as well. A different approach that we would like to compare with involves replacing $\N$ by a moduli space of complexes and a different stability condition; see recent work of \cite{B}.\\
Furthermore, we would like to show that the virtual structure on $\N_{Sp(r)}$ is again symmetric, so $[\N_{Sp(r)}]^{vir}$ is degree $0$ and we can set up new Vafa-Witten invariants using torus localization as in \cite{TT1}.\\
Finally, we remark that although under certain assumptions both $O(r)$ and $Sp(r)$ appear in $\N$, we restrict to $Sp(r)$ for this discussion. There are some stability issues with $O(r)$ and as $Sp(2)\cong SL(2)$, $\N_{Sp(2)}$ has already a virtual theory as discussed in \cite{TT1} and \cite{Sch}.
\subsection{Notation}
The bounded derived category of $\N$ is denoted as $\mathbf{D}^b(\N)$. 
If $\N$ is endowed with a group action $G$, we denote the bounded derived category of $G$-\textit{equivariant sheaves} as $\mathbf{D}^b(\N)^G$, see e.g. \cite{R} for an introduction to equivariant categories.\\ 
Furthermore, we use the notation $$\RHomm_{p_X}(\EE,\EE):= \mathbf{R}{p_{X,*}}(\mathbf{R}\Homm(\EE,\EE))\in \mathbf{D}^b(\N),$$
with canonical projection $ p_X: X\times \N \rightarrow \N$. Notation will be simplified whenever possible: For example, we will use for the map $\pi: X\rightarrow S$ and its base change $ X\times \N \rightarrow S \times \N$ the same notation. For the canonical line bundle $K_S$ on $S$, we use the notation $\pi^*K_S$ for the sheaf on $X$ and its pullback to $X\times \N$.
\subsubsection{Duals}
For coherent $\F$, we define its dual as $\F^*:= \Homm(\F,\OO)$. If $V^\bullet$ is a complex of coherent sheaves, we denote by $V^{\bullet,\vee}$ the dual complex. Furthermore, for spectral sheaves $\E_\phi$ on $X$, we have duals $\E_\phi^D:= \Extt^1(\E_\phi,\pi^*K_S^{-1})$ which we define similarly for families $\EE$. 
\subsection{Acknowledgments}
Special thanks to Richard Thomas for his ideas and encouragement during early stages of this project.\\
Furthermore, we want to thank Tasuki Kinjo, Woonam Lim and Wiktor Vacca for reading and helpful comments and Martijn Kool and Georg Oberdieck for inspiring discussions around the topic.

\section{Higgs pairs}    
\subsection{Higgs pairs}
   We fix a smooth projective surface $S$ over $\mathbf{C}$ with polarization $\OO_S(1)$. A Higgs pair $$(E,\phi) \; \text{on} \; S $$ is a coherent sheaf $E$ together with a map $$\phi \in \Gamma(\Endd(E)\otimes K_S)=\Hom(E,E\otimes K_S).$$
   If $X:=\mathrm{Tot}(K_S) \xrightarrow{\pi} S$ denotes the total space of the canonical bundle, the spectral correspondence $$\mathbf{Higgs}(S)\cong \mathbf{Coh}_c(S),$$ identifies pairs $(E,\phi)$ with their spectral sheaves $$\E_\phi \; \text{on} \; X,$$ which are compactly supported sheaves on $X$.\\
   Higgs pairs $(E,\phi)$ form an abelian category $\mathbf{Higgs}(S)$, where a morphism $$f:(E,\phi) \rightarrow (E',\phi')$$
   is a morphism $$f:E\rightarrow E'$$ that intertwines $\phi$ with $\phi'$, i.e. $$f\phi=\phi'f.$$
   Given $(E,\phi)$, we get $\E_\phi$ as the sheaf of eigenspaces of $\phi$ acting on $E$, supported on the corresponding eigenvalues in $K_S$.\\ 
   And given $\E_\phi$, we recover $$(E,\phi)=(\pi_*\E_\phi,\pi_*(\tau\cdot \id)),$$ where $\tau\cdot \id$ is the tautological endomorphism on $X$.\\ We refer to \cite[2.2]{TT1} for the details of the spectral construction.
\subsection{Resolutions}
There is a natural resolution of $\E_\phi$ in terms of $(E,\phi)$, see \cite[2.11.]{TT1}. Namely, if $\tau$ denotes the tautological section of $\pi^*K_S$, this is given by 
\begin{equation}
    \begin{tikzcd} [column sep=large]
        0 \arrow[r] & \pi^*E \otimes K_S^{-1} \arrow[r, "\pi^*\phi-\tau\cdot \id"]& \pi^*E \arrow[r,"ev"]& \E_\phi \arrow[r] & 0
    \end{tikzcd}
\end{equation}
\begin{rmk}
An important observation is that if $E$ is a bundle, $\E_\phi$ has \textit{projective dimension} $\mathrm{pd}(\E_\phi)$ equal to $1$. 
\end{rmk}
\subsection{Different gauge groups}
We start with some preliminaries:\\
For a locally free sheaf $E$ with dual $E^*=\Homm(E,\OO_S)$, we have a natural isomorphism $E^*\otimes E^* \cong \Homm(E,E^*)$ and taking duals $f \mapsto f^*$ defines an involution with $\pm 1$ eigensheaves
$$\curly{S}ym^2(E^*)\oplus \wedge^2 E^*.$$
Sections of the former are symmetric $f^*=f$ and section of the latter skew $f^*=-f$.\\
If $E$ is endowed with a non-degenerate pairing $$f: E\xrightarrow{\sim} E^*,$$ the triple $(E,f,\phi)$ is an orthogonal Higgs bundle or $O(r)$-bundle, if $f$ is symmetric and $$\phi \in \Gamma(\so(E)\otimes K_S)\subset \Gamma(\Endd(E)\otimes K_S)$$ is a section of the $K_S$-twisted Lie algebra bundle of skew-symmetric endomorphisms.\\
We call the triple a symplectic Higgs bundle or $Sp(r)$-bundle if $f$ is skew and $$\phi \in \Gamma(\spp(E)\otimes K_S) \subset \Gamma(\Endd(E)\otimes K_S)$$ is a section of the $K_S$-twisted Lie algebra bundle of symmetric endomorphisms.\\ 
We refer to App. \ref{orthosymplectic} for some background. 
\subsection{$SO(r)$}
We call an $O(r)$-bundle $(E,f,\phi)$ together with an \textit{orientation} $o: \OO\xrightarrow{\sim} \wedge^rE$ an $SO(r)$-bundle, see  \cite[Def.2.1]{OT}. For example, every $O(2)$-bundle splits into its isotropic lines $E\cong L\oplus M$. In this simplest case, choosing an orientation is the same as fixing an isomorphism $ M \xrightarrow{\sim} L^*$, such that $E\cong L \oplus L^*$ and $\wedge^2E\cong L\otimes L^*$ and the map $$\OO \xrightarrow{\id} L \otimes L^* \xrightarrow{a \otimes b \mapsto -b \otimes a} L^*\otimes L \xrightarrow{\ev} \OO$$ is given by $1 \mapsto -1$.  
\subsection{Reflexive sheaves}
    We have an exact sequence
    \begin{equation*}
         \begin{tikzcd}
        0 \arrow[r] & E^{tor} \arrow[r] &E  \arrow[r,"\ev"] & E^{**}, 
        \end{tikzcd}
    \end{equation*}
    where the first map $E^{tor} \subset E$ is the natural inclusion of the torsion sub-sheaf and the second the canonical evaluation map $e \mapsto \ev_e$. That means, $\ev:E \hookrightarrow E^{**}$ if $E$ is torsion free.
    We call $E$ reflexive, if $ev$ is an isomorphism. Clearly, this holds if $E$ is locally free. 
    \begin{dfn}
        Throughout this discussion,  $E$ will be a \textit{torsion free} sheaf $E$ of $\rk(E)=r$. 
    \end{dfn}
    Moreover, we recall (see e.g. \cite[\href{https://stacks.math.columbia.edu/tag/0B3N}{Tag 0B3N}]{SP}):
    \begin{lem}\label{reflexiveisbundle}
        Over a smooth surface $S$, $E$ is reflexive precisely if it is locally free.
    \end{lem}
    \begin{cor}\label{doubledualext}
        $E\hookrightarrow E^{**}$ extends uniquely to a map of Higgs pairs $(E,\phi) \hookrightarrow (E^{**},\phi^{**})$.
    \end{cor}
    \begin{proof}
        As $E^{**}$ is always reflexive, the lemma implies it is locally free. Furthermore, as both $E$ and $E^{**}$ are locally free in codimension $1$, $E^{**}/E$ is supported on points of $S$, so there is an open $U$ where $E|_U \cong E^{**}|_U$, together with an extension $\phi^{**}$ of $\phi$. As $\Endd(E^{**})\otimes K_S$ is again locally free, it is reflexive, thus it is a $S_2$-sheaf\footnote{For $S_2$ sheaves and extension of sections, we refer to 1.1.10 \cite{HL} and 1.17 of \cite{Schw}.} and $\phi^{**}$ extends uniquely to a map on all of $E^{**}.$
    \end{proof}
  \subsection{Stability}  
    \begin{dfn}
        Denoting $H:=c_1(\OO_S(1))$ on $S$, we define the slope of $E$ to be $$ \mu(E):= \frac{\deg(E)}{\rk(E)},$$
        where $\deg(E):=\int_S c_1(E).H$.\\
        We call the pair $(E,\phi)$ slope semi-stable or $\mu$-semi-stable, if 
        $$\mu(F) \leq \mu(E)$$
        for all proper $\phi$-invariant subsheaves $F$ on $S$ and we call $(E,\phi)$ slope stable, if the inequality is strict. 
    \end{dfn}
    \begin{dfn}
         The reduced Hilbert polynomial of $E$ is given by $$p_E(n):=\frac{P_E(n)}{\rk(E)}$$ and we call $(E,\phi)$ Gieseker semi-stable if we have $$p_F(n) \leq p_E(n) \; \text{for} \; n \gg 0,$$ 
         and we define Gieseker stability again by making the inequality strict.
    \end{dfn}
    \begin{rmk}
        If $c_1(E)=0$, $p_E$ is given by
    \begin{align}\label{HPc10}
        P_E(n)=r\chi(\OO_S)-c_2(E)+ \frac{1}{2}(rnHc_1(S)+r(r+1)n^2H^2),
    \end{align}
    which follows from simplifying the general formula of App. \ref{HP}.
    \end{rmk}
    We recall the following standard fact, see \cite[1.2.7]{HL}:
    \begin{lem}\label{stableiso}
        Let $f:(E,\phi) \rightarrow (E',\phi')$ be a non-zero map of semi-stable Higgs pairs with $p_{E}=p_{E'}$. Assume that $(E,\phi)$ is stable. Then $f$ is an isomorphism.  
    \end{lem}
    \begin{claim}\label{doubledual}
        If $(E,\phi)$ is stable, so is $(E^{**},\phi^{**})$. 
    \end{claim}
    \begin{proof}
        As $c_1(E)=c_1(E^{**})=0$, we have $\mu(E)=\mu(E^{**})=0$. \\ Assuming that $(F,\psi)\subset (E^{**},\phi^{**})$ is a Higgs sub-sheaf, pulling back along $\ev$ gives a Higgs sub-sheaf $(E\cap F, \psi|_{E\cap F})$ of both $(F,\psi)$ and $(E,\phi)$. \\
      Thus, as $E^{**}/E$ is supported on points, so is $E/(E\cap F)$: we have $c_1(F).H=c_1(E\cap F).H<0$, as $(E,\phi)$ is stable, hence $\mu(F)<0$ as desired. 
    \end{proof}
    \begin{rmk}
        We further remark that slope stability implies Gieseker stability, which implies slope semi-stability. Thus, if the invariants $\ch(E)$ are chosen such that slope stability agrees with slope semi-stability, we get that Gieseker stability = slope stability. 
    \end{rmk}
    \subsection{Moduli spaces}\label{moduli}
    Fixing invariants $\ch(E)$ such that slope semi-stability implies stability, there is a quasi-projective moduli scheme $$\N=\N_{\ch(E)}$$ whose closed points $\mathbf{Spec}(\C) \rightarrow \N$ parametrize isomorphism classes of stable Higgs pairs $[(E,\phi)]$. As stable sheaves are simple, we have $$\mathrm{Aut}(E,\phi)\cong \mathbf{C}^\times.$$
    We fix $\ch(E)=(r,0,c_2)$ such that semi-stability implies stability.
    \begin{claim}
        Let $c_2(E)$ be of minimal degree such that $\N$ is non-empty. Then all pairs $(E,\phi)$ are locally free.
    \end{claim}
    \begin{proof}
        $E$ is torsion free, we have a canonical embedding $E \hookrightarrow E^{**}$, extending uniquely to a map $(E,\phi)\hookrightarrow (E^{**},\phi^{**})$ by Cor. \ref{doubledualext}, which is an isomorphism if and only if $E$ is locally free by Lem. \ref{reflexiveisbundle}.\\
        Assuming that $E$ is not locally free, $\rk(E)=\rk(E^{**})$ implies that $P_{E}< P_{E^{**}}$ by stability of $(E,\phi)$. By the above expression for Hilbert polynomials \ref{HPc10}, this implies $c_2(E)> c_2(E^{**})$, which contradicts minimality of $c_2(E)$.
    \end{proof}    
\section{Involutions}
We want to express pairs $(E,\phi)$ of $\N$ of gauge group $Sp(r)$ and $O(r)$ as fixed points of an involution. To make this into a well-defined action on moduli spaces, we need to address stability. To keep it simple, the involution will be first defined for single sheaves, then in families. 
\subsection{Single sheaves}
    \begin{dfn}
         Fixing a Chern character $\mathrm{ch}(E)=(r,0,c_2)$ with $r>0$  and $c_2(E)$ minimal as above, we define $$ (E,\phi) \mapsto (E^*,-\phi^*),$$
        preserving $\ch(E)$.
    \end{dfn}
    We claim that this preserves stability:
    \begin{claim}
       If $(E,\phi)$ is $\mu$-stable, so is $(E^*,- \phi^*)$.
    \end{claim}
    \begin{proof}
        Note that $\mu(E)=\mu(E^*)=0$. If $(E,\phi)$ is stable, let $(E^*,-\phi^*) \twoheadrightarrow (Q,\psi)$ be a Higgs quotient, so we need to show that $\deg(Q) >0$. By \cite[1.2.6]{HL} we may assume that $Q$ is torsion free. As $E$ is reflexive, $ (Q^*,-\psi^*)\subset (E,\phi)$ is a Higgs sub-sheaf, so $\deg(Q^*) <0$ by stability, thus $\deg(Q)=-\deg(Q^*)>0$.
    \end{proof}   
\subsection{Universal families}
We have a universal family of Higgs bundles $$(\mathsf{E},\Phi) \; \text{on} \; S\times \N,$$
flat over $\N$.
Equivalently, this is a $\N$-flat universal family of spectral sheaves $$\EE \; \text{on} \; X\times \N.$$
So in particular, for a closed point $n:\Spec(\C) \rightarrow \N$, corresponding to $\E_\phi$ on $X$, we have $\EE|_{X_n}\cong \E_{\phi}$. \\
As $\strE$ is a bundle, there is a two-term resolution of locally frees (\cite[2.11]{TT1}),
$$\EE^\bullet=[\EE^{-1}\xrightarrow{d} \EE^0] \twoheadrightarrow \EE$$
We further remark that universal families always exists locally and globally perhaps only as a \textit{twisted} universal family.\\
Furthermore, the deformation-obstruction complex $$\RHomm_{p_X}(\EE,\EE)[1] \in \mathbf{D}^b(\N)$$
is always a well-defined object.

\subsection{Classifying maps}
As stability of $(E,\phi)$ implies stability of the pair $(E^*,-\phi^*)$,
the functor
$$\iota: (\mathsf{E},\Phi) \mapsto (\mathsf{E}^*,-\Phi^*),$$
preserves flatness and induces a classifying map $$\iota: \N \rightarrow \N,$$
such that $$(\id\times \iota)^*(\mathsf{E},\Phi)\cong(\mathsf{E}^*,-\Phi^*).$$
Thus, $\iota$ is point-wise given by $$ [(E,\phi)] \mapsto [(E^*,-\phi^*)].$$
As bundles are reflexive, we immediately see that $\iota^2=\id$ holds.
\section{The fixed locus}\label{fixedlocus}
    \subsection{Fixed points}
    In this section, we describe fixed points of $$\iota: (E,\phi) \mapsto (E^*,-\phi^*).$$
    
    As $\iota$ is a well-defined map on $\N$, we may consider its scheme-theoretic fixed locus $\N^\iota \subset \N$.\\
    Such a fixed point $[(E,\phi)] \in \N^\iota$ is given by an isomorphism $f:E\xrightarrow{\sim} E^*$ intertwining $\phi$ and $-\phi^*$, i.e. there is a commutative square
    \begin{center}
        \begin{tikzcd}
            E \arrow[r,"f"]\arrow[d,"\phi"] &  E^* \arrow[d,"-\phi^*"] \\
            E\otimes K_S \arrow[r,"f\otimes1"] & E^*\otimes K_S.
        \end{tikzcd}
    \end{center}
    As stable sheaves are simple, $\mathrm{Aut}(E,\phi)=\langle f \rangle $, so $f$ determines the automorphisms of $(E,\phi)$.
\subsection{In families}
    In terms of families $(\strE,\Phi)$ of Higgs bundles, there is a linearization map over the fixed locus $X\times\N^\iota$, $$\tilde{f}:(\strE,\Phi) \xrightarrow{\sim} (\mathsf{E}^*,-\Phi^*)$$ and we split $$\tilde{f}=\tilde{q}+\tilde{\omega} \in \Gamma(\Homm(\strE,\strE^*)|_{X\times\N^\iota})\cong\Gamma(\curly{S}ym^2(\strE^*)|_{X\times \N^\iota})\oplus \Gamma(\wedge^2\strE^*|_{X\times \N^\iota})$$ into symmetric and skew-symmetric part.
\subsection{Determinants}
    If $E \xrightarrow{\sim} E^*$ is a pairing, then taking determinants gives that $\det(E)$ is $2$-torsion. There is a determinant map\footnote{For background on (derived) determinant maps, we refer to \cite{STV}.} $$\det: \N \rightarrow \mathbf{Pic}^0(S),$$ sending $$E\mapsto \det(E),$$ thus we see that the restriction  to $\N^\iota$ gives a map $$\N^\iota \rightarrow \mathbf{Pic}^0(S)[2]\subset \mathbf{Pic}^0(S)$$ mapping to the discrete set of $2$-torsion line bundles. As we are only interested in sheaves $E$ with $\det(E)\cong \OO_S$, we choose the component of $\N^\iota$ given by $\N^\iota \cap \det^{-1}([\OO_S])$.\\
    For instance, if $r=2$, this restriction

    For the rest of this discussion, this will be dropped from notation and we may simply assume all $\iota$-fixed bundles have trivial determinants.
    \begin{prop}\label{fixedlocuscomp}
        If $(E,\phi)$ is a $\iota$-fixed Higgs bundle, then $f$ is either symmetric $f=f^*$ or skew $f=-f^*$ and $(E,\phi)$ admits either an orthogonal or symplectic structure.\\
        Accordingly, we have two components of the moduli fixed locus $$\N_{{O}(r)} \hookrightarrow \N^\iota \hookleftarrow \N_{Sp(r)}.$$
    \end{prop}
    \begin{proof}
        Let $f:(E,\phi)\xrightarrow{\sim} (E^*,-\phi^*)$ be a $\iota$-fixed Higgs bundle. We write $f=q +\omega$ as symmetric and skew and first assume that $q\neq 0$.
        We get $$q\phi + \omega \phi=-\phi^*q-\phi^*\omega,$$
        so taking duals of this equation and adding it to both sides results in $$q\phi=-\phi^*q.$$
        Thus $$q: (E,\phi) \rightarrow (E^*,-\phi^*)$$ is a non-zero map, so by Lem. \ref{stableiso}, $q$ is invertible and defines a non-degenerate pairing $q:E\xrightarrow{\sim} E^*$. Furthermore, we compute $$(q\phi)^*=\phi^*q^*=\phi^*q=-q\phi,$$
        so $q\phi$ is skew, therefore $\phi\in \Gamma(\so(E)\otimes K_S)$. As $\mathrm{Aut}(E,\phi)\cong\C^\times$, we must have $\omega=0$.\\
        Conversely, for every orthogonal Higgs bundle $(E,q,\phi)$, $q\phi$ is skew (see App. Cor. \ref{orthogonalisskew}), thus $q\phi=-(q\phi)^*=-\phi^*q$ implies that it gives a $\iota$-fixed point.
        Thus, $$\N_{{O}(2n)}=\N^\iota \cap\{\tilde{q}\neq 0\}=\N^\iota \cap\{\tilde{\omega}=0\}\subset \N^\iota$$ is both open and closed. \\
        If $q=0$, then $f=\omega$ and $(E,\omega,\phi)$ is a symplectic Higgs bundle: Indeed, as we have $\omega \phi=-\phi^*\omega$, we find $$(\omega \phi)^*=\phi^*\omega^*=-\phi^*\omega=\omega \phi,$$ so $\omega\phi$ is symmetric, hence $\phi \in \Gamma(\spp(E)\otimes K_S)$. For the same reason as before, any symplectic bundle $(E,\omega,\phi)$ lies in $\N^\iota$ and we get $$\N_{{Sp}(2n)}= \N^\iota \cap \{\tilde{q}=0\}=\N^\iota \cap \{\tilde{\omega} \neq 0\}$$ is open and closed, so another component of $\N^\iota$.
    \end{proof}
\begin{rmk}
We conclude that orthosymplectic Higgs bundles appear in $U(r)$-moduli spaces $\N$  as fixed points of an involution.\\
The virtual structures are induced by their deformation-obstruction theory, which we want to phrase in terms of sheaves on $X$. For this we need to formulate $\iota$ in terms of spectral sheaves $\E_\phi$. 
\end{rmk}
\begin{dfn}
We decompose $\iota=D\circ \sigma$ into dualizing $$ D: (E,\phi) \mapsto (E^*,\phi^*)$$
and the sign map $$\sigma: (E,\phi) \mapsto (E,-\phi)$$
\end{dfn}
\section{Action on spectral sheaves}\label{spectral}
We start with some preliminaries.
Both maps, $D$ and $\sigma$, act on the resolution 
\begin{equation}\label{spectralresol}
    \begin{tikzcd} [column sep=large]
        0 \arrow[r] & \pi^*E \otimes K_S^{-1} \arrow[r, "\pi^*\phi-\tau\cdot \id"]& \pi^*E \arrow[r,"ev"]& \E_\phi \arrow[r] & 0
    \end{tikzcd}
\end{equation}
in a natural way and give us an explicit description of $\E_{-\phi^*}$.
\subsection{Duals} We start with the correct notion of duals for spectral sheaves:
    \begin{dfn}
        We define the dual of $\E_\phi$ to be the sheaf $$\E_\phi^D:=\Extt^1(\E_\phi,\pi^*K_S^{-1}).$$ 
    \end{dfn}
    \begin{lem}\label{spectralprojdim}
        If $E$ is locally free, we have $$\RHomm(\E_\phi,\pi^*K_S^{-1})[1]\cong \E_\phi^D.$$
        Furthermore, $$\E_\phi^D=\coker(\pi^*\phi^*-\tau \cdot \id)\otimes \pi^*K_S^{-1}$$
    \end{lem}
    \begin{proof}
        As $E$ is locally free, the resolution \ref{spectralresol} implies that $\mathrm{dh}(\E_\phi)=1$. As $\mathrm{codim}(\E_\phi,X)=1$, $\E_\phi$ is an $S_2$-sheaf, thus $\Extt^i(\E_\phi,\OO_X)=0$ for $i>1$ by \cite[1.1.10]{HL}. As $\E_\phi$ is pure torsion, $\Homm(\E_\phi,\OO_X)=0$ and we get $\RHomm(\E_\phi,\OO_X)[1]\cong \Extt^1(\E_\phi,\OO_X)$. Tensoring both sides with $\pi^*K_S^{-1}$ shows the first equality.\\
        For the second, we apply $\RHomm(\_,\OO_X)$ to \ref{spectralresol} and recall that for all bundles $F$, $\Extt^{i}(F,\OO)$ vanish for $i>0$, thus $$\coker(\pi^*\phi^*-\tau\cdot \id)=\Extt^1(\E_\phi,\OO_X)\otimes \pi^*K_S.$$
    \end{proof}
    \begin{cor}
       The spectral sheaf $\E_{\phi^*}$ associated to $(E^*,\phi^*)$ is given by $$\E_\phi^D=\Extt^1(\E_\phi,\pi^*K_S^{-1}).$$
    \end{cor}
    \subsection{$(-1)$ on fibers}
We add the action $(E,\phi)\mapsto (E,-\phi)$, starting again with single sheaves $\E_\phi$: we recover $\phi=\pi_*(\tau \cdot \id)$, where $\tau$ is the tautological section of $\pi^*K_S$, i.e. $\tau_{(s,t)}=t$ for local coordinates $(s,t)$ of $X$.\\
     "$(-1)$ on fibers" is the map given by
    \begin{dfn}
        $$\sigma: X\rightarrow X,$$
        $$(s,t) \mapsto (s,-t).$$
    \end{dfn}
    \begin{lem}
        $\sigma^*\E_\phi=\E_{-\phi}$ 
    \end{lem}
    \begin{proof}
        We show that applying $\sigma^*$ to the resolution \ref{spectralresol}
        \begin{equation*}
        \begin{tikzcd} [column sep=large]
        0 \arrow[r] & \pi^*E \otimes K_S^{-1} \arrow[r, "\pi^*\phi-\tau\cdot \id"]& \pi^*E \arrow[r,"ev"]& \E_\phi \arrow[r] & 0
        \end{tikzcd}
        \end{equation*}
        yields $\E_{-\phi}$.\\
        It is clear that 
        $\sigma^*$ preserves exactness. And as $\sigma^*\tau_{(s,t)}=\tau_{(s,-t)}$, we have $\pi_*(\sigma^*\tau \cdot \id)=-\phi$. Thus,
        \begin{align*}
           &\sigma^*\coker(\pi^*\phi-\tau\cdot \id)=\coker(\pi^*\phi-\sigma^*\tau\cdot \id) \\
           &= \coker(\pi^*\phi+\tau\cdot \id)=\coker(\pi^*(-\phi)-\tau\cdot \id)= \E_{-\phi} 
       \end{align*}
       gives us the result.
    \end{proof}
    \begin{cor}
        The spectral sheaf $\E_{-\phi}$ of $(E,-\phi)$ is given by $\sigma^*\E_\phi$.
    \end{cor}
    \section{In families}
    In order to discuss the deformation-obstruction theory, we need to
    phrase above maps for families $\EE$ of spectral sheaves:
    \subsection{Dualizing}
        Let $\EE^\bullet=[\EE^{-1}\xrightarrow{d} \EE^0]$ be a locally free resolution for $\EE$.     
    \begin{claim}\label{claim}
         $\EE^{\bullet, \vee}\otimes \pi^*K_S^{-1} $ is one for $\EE^D=\Extt^1(\EE,\pi^*K_S^{-1})$.
    \end{claim}
    \begin{proof}
        We have an exact sequence
        $$
        \begin{tikzcd}
            0 \arrow[r] & \EE^{-1} \arrow[r,"d"] &\EE^0 \arrow[r] &\EE \arrow[r] & 0,
        \end{tikzcd}
        $$
    where $\EE^{i}$ is a locally free sheaf.
    Applying $\RHomm(\_,\OO)$ and noting that $\Homm(\EE,\OO)=\Extt^{i>1}(\EE,\OO)=0$ gives 
    $$
        \begin{tikzcd}
        0 \arrow[r] & \EE^{0,\vee} \arrow[r,"d^\vee"] &\EE^{-1,\vee} \arrow[r] &\Extt^1(\EE,\OO) \arrow[r] & 0 .
    \end{tikzcd}
   $$
   Tensoring by $\pi^*K_S^{-1}$ preserves exactness and gives the desired resolution $ \EE^{\bullet,\vee} \otimes \pi^*K_S^{-1} \rightarrow \EE^D$. 
\end{proof}
    Thus, the functor $$\EE^\bullet \mapsto \EE^{\bullet,\vee} \otimes \pi^*K_S^{-1}$$ preserves flatness and defines a classifying map $$D: \N \mapsto \N,$$
    where we recover on closed points above dualizing action $$[\E_\phi] \mapsto [\E_{\phi}^D].$$
    \subsection{$(-1)$ on fibres}
    For the same reason, $$\EE^\bullet \mapsto (\sigma\times \id)^*\EE^\bullet $$
    defines a classifying map $$\sigma: \N \rightarrow \N, $$
    on closed points given by $$\sigma: [\E_\phi] \mapsto [\sigma^*\E_\phi].$$\\
    By the property of universal families, we get an isomorphism $$(\sigma\times \id)^*\EE^\bullet\cong (\id\times \sigma)^*\EE^\bullet$$ over $X\times \N$ and, after applying $(\sigma\times \id)^*$ again on both sides, the tautological identification 
    $$\EE \cong (\sigma \times \sigma)^*\EE,$$
    where by abuse of notation, $\sigma$ is  $(s,t)\mapsto (s,-t)$ on the first factor and $[\E_\phi] \mapsto[ \sigma^*\E_\phi]$ on the second.
    \begin{dfn}\label{sigmalinear}
        To keep everything simple, we denote the resulting map as $$\sigma_*: \EE^\bullet \cong \sigma^*\EE^\bullet; \; a \mapsto \sigma_*a.$$
        Taking duals and using that $\sigma^{*,2}=\id$, this is equivalently given as $\sigma_*: \EE^\vee \cong \sigma^*\EE^\vee.$
    \end{dfn}
    \subsection{$\iota$ for spectral sheaves}
        The invariants $\ch(E)$ of the moduli of stable Higgs bundles $\N$ fix those for the spectral sheaves $\ch(\E_\phi)$. And since stability of $(E,\phi)$ coincides with the one of $\E_\phi$ (see \cite[2.9]{TT1}), $\N$ is equivalently a moduli for stable spectral sheaves $\E_\phi$ on $X$.    
    \begin{dfn}
        Accordingly, the involution $$\iota: \N \rightarrow \N$$
        is in terms of spectral sheaves given as
        $$\E_\phi \mapsto \Extt^1(\sigma^*\E_\phi, \pi^*K_S^{-1}),$$
        and we sometimes write $\iota(\E_\phi)=\E_{-\phi^*}$.
    \end{dfn}
    Combining everything from this section, we find:
    \begin{prop}\label{spectralinvolution}
        The spectral sheaf of $(E^*,-\phi^*)$ is given by $$\E_{-\phi^*}=\Extt^1(\sigma^*\E_\phi,\pi^*K_S^{-1}).$$
    \end{prop}
    Prop. \ref{spectralinvolution} implies immediately:
    \begin{cor}
        There is a commutative square relating the two notions of $\iota:$
        \begin{equation}\label{diagraminvolution}
	   \begin{tikzcd}
	       X\times \mathcal{N} \arrow[r,"\id\times \iota"]\arrow[d,"\pi \times \id"] & X\times \mathcal{N} \arrow[d,"\pi \times \id"]\\
	       S\times \mathcal{N} \arrow[,r,"\id\times\iota"] &S\times \mathcal{N}. \\
	   \end{tikzcd}
	   \end{equation}
        \end{cor}
        \begin{proof}
        We are left to show that $\sigma^*\E_\phi^D$ is the spectral sheaf for $(E^*,-\phi^*)$. This can be seen by applying Grothendieck-Verdier duality to $\pi$, which gives $$\R\pi_*\RHomm(\sigma^*\E_\phi,\pi^*K_S^{-1})[\dim(\pi)]\cong \RHomm(\R \pi_*\sigma^* \E_\phi,\OO_S).$$ 
        By Lem. \ref{spectralprojdim}, we have that $\RHomm(\sigma^*\E_\phi,\pi^*K_S^{-1})[1]\cong \Extt^1(\sigma^*\E_\phi,\pi^*K_S^{-1})$ and $\RHomm(E,\OO_S)=E^*$. 
        As $\pi_*$ is affine, $\pi\circ \sigma=\pi$ and $\dim(\pi)=1$, the duality simplifies to $$\pi_*(\Extt^1(\E_\phi,\pi^*K_S^{-1}))\cong E^*.$$
        Furthermore, as the action of the tautological section $(\tau\cdot \id)$ on $\E_\phi^D$ pushes down to $\phi^*$, we get that $\pi_*(\sigma^*(\tau\cdot \id))=-\phi^*$.
        \end{proof}
    \section{The virtual differential}\label{virtualdiff}
    We split again $\iota=\sigma \circ D$ and deal with each map separately. We start with $D$.
    \subsection{Duals}
    Recalling $$ \E_\phi \mapsto \E^D:= \Extt^1(\E_\phi,\pi^*K_S^{-1}),$$
    we claim:
    \begin{claim}
     There is a lift of $D$ to a map $$\mathbf{R}\Homm_{p_X}(\EE,\EE)[1] \rightarrow \mathbf{R}\Homm_{p_X}(\EE^D,\EE^D)[1],$$
     compatible with the Atiyah class.
    \end{claim}
    
\begin{rmk}\label{rmk}
    This gives a locally free resolution $$\EE^{\bullet, \vee}\otimes^L \EE^\bullet \rightarrow \RHomm(\EE,\EE),$$ where we assume that $H^{i}(X_n,\EE^{\bullet, \vee}\otimes^L \EE^\bullet|_{X_n})=0$ for $i>0$. Then their higher direct images $\mathbf{R}^{i}p_{X,*}$ vanish and their push-downs to $\N$ are again locally free.
\end{rmk}
\begin{cor}
    $\RHomm(\EE^D,\EE^D)$ is represented by $\EE^{\bullet}\otimes ^L \EE^{\bullet, \vee}$.
\end{cor}
As $\RHomm(\EE,\EE)$ is represented by $\EE^{\bullet,\vee}\otimes^L \EE^\bullet$,  above Lem. \ref{claim} tells us that $\RHomm(\EE^D,\EE^D)$ is represented by $$(\EE^{\bullet, \vee}\otimes \pi^*K_S^{-1})^{\vee} \otimes^L (\EE^{\bullet, \vee}\otimes \pi^*K_S^{-1})\cong \EE^{\bullet,\vee, \vee }\otimes \pi^*K_S \otimes^L \EE^{\bullet,\vee}\otimes \pi^*K_S^{-1} \cong \EE^\bullet \otimes^L \EE^{\bullet,\vee},$$
where we used the natural maps $\pi^*K_S\otimes \pi^*K_S^{-1} \cong \OO$ and $\EE^\bullet \cong \EE^{\bullet, \vee,\vee}$.\\
\begin{dfn}
There is a dual action $$\RHomm(\EE,\EE) \rightarrow \RHomm(\EE^D,\EE^D)$$
$$f \mapsto -f^\vee.$$
In terms of the above presentations,
this is swapping the factors
\begin{align*}
        \EE^{\bullet,\vee}\otimes^L \EE^\bullet \rightarrow  \EE^{\bullet}\otimes^L \EE^{\bullet,\vee},
    \end{align*}
    \begin{align*}
        a \otimes b \mapsto -b \otimes a,
    \end{align*}
\end{dfn}
    Thus, applying $\mathbf{R}p_{X,*}$ to the dual action (swapping the factors)
    gives a morphism $$-(\_)^\mathfrak{\vee}: \mathbf{R}\Homm_{p_X}(\EE,\EE)[1] \rightarrow \mathbf{R}\Homm_{p_X}(\EE^D,\EE^D)[1], $$
    which we call the virtual differential of $D$ (omitting the shift from notation).  
Furthermore, this map is compatible with $D_*: \mathbf{T}_\N \rightarrow D^*\mathbf{T}_\N$ via partial Atiyah $\At_{\EE,\N}$:
\begin{claim}\label{Deq}
The following diagram commutes
    \begin{equation*}
    \begin{tikzcd}
	   [column sep=14ex]
	   \RHomm_{p_X}(\EE,\EE)[1] \arrow[r,"-(\_)^\vee"] &\RHomm_{p_X}(\EE^D,\EE^D)[1]\\
	   \mathbf{T}_{\N}\arrow[u,"\At_{\EE,\N}"] \arrow[r,"D_*"] & D^*\mathbf{T}_{\N} \arrow[u,"D^*\At_{\EE,\N}"].
	\end{tikzcd}
\end{equation*}
\end{claim}

\begin{proof}
    It is a standard fact that $\At_{\EE^D,\N}=-(\At_{\EE,\N})^\vee$ and the claim follows from functoriality of Atiyah classes, as $$(D^*\At_{\EE,\N})\circ D_*=\At_{D^*\EE,\N}=\At_{\EE^D,\N}.$$ 
\end{proof}
\subsection{$(-1)$ on fibres}
    We continue with lifting $\sigma$ to $\RHomm_{p_X}(\EE,\EE)[1]$.\\
    As discussed, $\EE \cong \sigma^*\EE$ defined in  \ref{sigmalinear} induces an isomorphism $$\EE^\vee\otimes^L \EE \xrightarrow{\sim} \sigma^*\EE^\vee\otimes^L \sigma^*\EE.$$ Pushing down to $\N$ lifts $\sigma$ to an action on the virtual tangent complex $$ \sigma_*: \RHomm_{p_X}(\EE,\EE)[1] \rightarrow \RHomm_{p_X}(\sigma^*\EE,\sigma^*\EE)[1].$$
    This is compatible with Atiyah classes, as seen it \cite[6.5.]{Sch}. Here, $\sigma$ defines a different action on the Higgs field, but this does not affect the proof in any way. Alternatively, this can be seen as a special case of the $\mathbf{C}^\times$-equivariance of the Atiyah class, as proven in \cite[4.3]{R}.
    \begin{claim}\label{sigmaeq}
        $\sigma_*$ is compatible with the differential map $\sigma_*: \mathbf{T}_\N \rightarrow \sigma^*\mathbf{T}_\N$ via Atiyah classes, i.e. 
        \begin{equation} 
	\begin{tikzcd}
	[column sep=14ex]
	\mathbf{R}\Homm_{p_X}(\EE,\EE)[1] \arrow[r,"\sigma_*"] &\mathbf{R}\Homm_{p_X}(\sigma^*\EE,\sigma^*\EE)[1]\\
	\mathbf{T}_{\N}\arrow[u,"\At_{\EE,\N}"] \arrow[r,"\sigma_*"] & \sigma^*\mathbf{T}_{\N} \arrow[u,"\sigma^*\At_{\EE,\N}"]
	\end{tikzcd}
	\end{equation}	
    commutes.
    \end{claim}
\section{Equivariant sheaves}
\subsection{Equivariant objects}
In order to describe the virtual structure on $\N^\iota$, we need a $\iota$-equivariant structure on the Atiyah class $$(\At_{\EE,\N}: \mathbf{T}_\N \rightarrow \RHomm_{p_X}(\EE,\EE)[1]) \in \mathbf{D}^b(\N)^{\langle \iota \rangle}$$ We briefly recall the notion of equivariant complexes on $\N$ in this concrete setup for any involution $\iota$ on $\N$.

\begin{dfn}\label{equivdef}
	We call $\At_{\EE}$ $\iota$-equivariant if there is a commutative diagram 
	\begin{equation} 
	\begin{tikzcd}
	[column sep=14ex]
	\mathbf{R}\Homm_{p_X}(\EE,\EE)[1] \arrow[r,"\iota_*"] &\mathbf{R}\Homm_{p_X}(\iota^*\EE,\iota^*\EE)[1]\\
	\mathbf{T}_{\N}\arrow[u,"\At_{\EE}"] \arrow[r,"\iota_*"] & \iota^*\mathbf{T}_{\N} \arrow[u,"\iota^*\At_{\EE}"]
	\end{tikzcd}
	\end{equation}	
	such that the triangle of differentials
	\begin{center}
		\begin{tikzcd}
		\mathbf{T}_\N \arrow[r,"\iota_*"] \arrow[dr,equal] & \iota^*\mathbf{T}_\N \arrow[d,"\iota^*(\iota_*)"]\\
		& (\iota^{*,2})\mathbf{T}_\N \\
		\end{tikzcd}
	\end{center}
    map to the virtual differentials
	\begin{center}
		\begin{tikzcd}
		\mathbf{R}\mathcal{H}om_{p_X}(\EE,\EE)[1] \arrow[r,"\iota_*"] \arrow[dr,equal]&
		\mathbf{R}\mathcal{H}om_{p_X}(\iota^*\EE,\iota^*\EE)[1] \arrow[d,"\iota^*(\iota_*)"]\\
		&\mathbf{R}\mathcal{H}om_{p_X}(\iota^{2,*}\EE,\iota^{2,*}\EE)[1] \\
		\end{tikzcd}
	\end{center}
	via $\At_{\EE}$ and $\iota^*\At_{\EE}$. Combining the claims \ref{sigmaeq} and \ref{Deq} we conclude:
\end{dfn}
    \begin{cor}
        $\At_{\EE,\N}$ is equivariant with respect to both actions $\sigma$ and $D$.
    \end{cor}
    We are left to show the same holds for $\iota=D\circ \sigma$.
    \begin{lem}\label{commute}
        We claim that the the lifts of $\sigma, D$ both to $\mathbf{T}_\N$ and $ \mathbf{T}^{vir}_\N$ commute, i.e. $$D_*\circ \sigma_* = \sigma_* \circ D_* $$ 
    \end{lem}
    \begin{proof}
        This is clear for the action on $\mathbf{T}_\N$. For the virtual action, we claim that the following diagram commutes:
        \begin{equation} 
	       \begin{tikzcd}
	       [column sep=14ex]
	       \EE^{\bullet,\vee}\otimes^L \EE^\bullet \arrow[d,"-(\_)^{\vee}"] \arrow[r,"\sigma_*"] &\sigma^*\EE^{\bullet,\vee}\otimes^L \sigma^*\EE^\bullet \arrow[d,"-(\_)^{\vee}"] \\
	       \EE^{\bullet}\otimes^L \EE^{\bullet,\vee} \arrow[r,"\sigma_*"] & \sigma^*\EE^{\bullet} \otimes^L \sigma^*\EE^{\bullet,\vee}
	       \end{tikzcd}
        \end{equation}	
        Indeed, following the RHS gives $$a\otimes b \mapsto \sigma_*a\otimes \sigma_* b\mapsto -\sigma_*b \otimes \sigma_*a =\sigma_*(-(a\otimes b)^\mathfrak{t}),$$
        which agrees the LHS composition.
        The statement follows by applying $\mathbf{R}p_{X,*}$ to the diagram.
    \end{proof}
    \begin{dfn}
    We define the virtual action of $\iota$ to be $$\iota_*= D_*\circ\sigma_*: \RHomm_{p_X}(\EE,\EE)[1] \rightarrow \RHomm_{p_X}(\iota^*\EE,\iota^*\EE)[1]$$
    \end{dfn}
    \begin{prop}
        $\At_{\EE,\N}$ is $\iota$-equivariant in the above sense.
    \end{prop}
    \begin{proof}
        As we have shown the equivariance for both $D_*$ and $\sigma_*$, 
        it is  clear that both actions $\iota_*$ - the one on $\mathbf{T}_\N$ and the one on $\RHomm_{p_X}(\EE,\EE)[1]$ - are compatible via $\At_{\EE,\N}$ and square to the identity, as $$\iota_*^2=(D_*\circ \sigma_*)^2=(D_*\circ \sigma \circ D_* \circ \sigma_*)^2=D_*^2\circ \sigma_*^2=\id,$$
        by \ref{commute}.
    \end{proof}

   % we need to show that
    %\begin{itemize}
    %    \item for $\sigma, D$, there are lifts $\sigma_*,D_*$ to $\RHomm_{p_X}(\EE,\EE)[1]$, compatible with $\At_{\EE,\N}$ and that
     %   \item $\sigma_*\circ D_*=D_* \circ \sigma_*$ holds.
   % \end{itemize}
    % Then, $\iota_*^2=(D\circ \sigma)^2_*=D_*\circ \sigma_*\circ D_* \circ \sigma_*=D_*\circ D_*\circ \sigma_*\circ \sigma_*=\id$.
    %Having proved the first point in ** and **, we are left to prove the second. This is clear for the canonical action on $\mathbf{T}_\N$. For the virtual differentials, we claim:
    %\begin{claim}
     %   $D_*\circ \sigma_*= \sigma_* \circ D_*$ as maps $$ \mathbf{R}\mathcal{H}om_{p_X}(\EE,\EE)[1] \rightarrow
	%	\mathbf{R}\mathcal{H}om_{p_X}(\iota^*\EE,\iota^*\EE)[1].$$
    %\end{claim}
    %\begin{proof}
     %   The isomorphisms $$ \EE \cong \sigma^*\EE$$ induce $$ \EE$$
    %\end{proof}

\section{The localization formula}
We define a virtual structure for $Sp(n)$-Higgs bundles by identifying them as fixed points $$\N_{Sp(r)}\hookrightarrow \N^\iota \subset \N$$ and apply equivariant localization as in \cite{GP} and \cite{Sch}. This requires
\begin{itemize}
    \item a $\iota$-equivariant embedding $\N\subset \curly{A}$ into a smooth ambient space.
    \item a $\iota$-equivariant Atiyah class $\At_{\EE,\N}$ represented by complexes.
\end{itemize}
\subsection{Smooth ambient space}
We remark that $\N$ admits an embedding into a smooth ambient space $\N \subset \curly{A}$ that is compatible with $\iota$, see \cite[1.2]{Sch}.\\
If $ \curly{I} \subset \OO_\curly{A}$ denotes the ideal sheaf of this embedding, then this gives a $2$-term equivariant representation of Illusie's truncated cotangent complex $$\mathbf{L}_\N=[\curly{I}/\curly{I}^2 \rightarrow \Omega_{\curly{A}}|_{\N}] \in \mathbf{D}^{[-1,0]}({\N})^{\langle \iota \rangle},$$
see \cite[3.2]{R}.
\subsection{$\iota$-equivariant representation.}
From the previous section after taking duals, we have an equivariant obstruction theory on $\N$ given by $$(\At_{\EE,\N}: \RHomm_{p_X}(\EE,\EE)[2] \rightarrow \mathbf{L}_\N )\in \mathbf{D}^b(\N)^{\langle \iota \rangle}.$$
This obstruction theory can be made perfect, i.e. there is a $2$-term representation of locally frees
$$V^\bullet=[V^{-1}\rightarrow V^0] \rightarrow \tau^{[-1,0]}\RHomm_{p_X}(\EE,\EE)[2],$$
which gives the $U(r)$-perfect obstruction theory on $\N$
$$V^\bullet \xrightarrow{\psi} \mathbf{L}_\N,$$
see \cite[3.18]{TT1}\\
By our previous considerations, we may choose $V^\bullet$ to be an equivariant representation with respect to $\iota$, which together with the equivariance of $\At_{\EE,\N}$ and representation of $\mathbf{L}_\N$ above gives a $\iota$-equivariant map of $2$-term complexes 
\begin{equation}
    ([V^{-1}\rightarrow V^0] \xrightarrow{\psi} [\curly{I}/\curly{I}^2 \rightarrow \Omega_{\curly{A}}|_\N]) \in \mathbf{D}^{[-1,0]}(\N)^{\langle \iota \rangle},
\end{equation}
where we refer to \cite[10.1]{Sch} for details of the equivariant representation.
\subsection{Equivariant localization.}
We adapt equivariant localization of \cite{GP}, replacing $\C^\times$ by $\langle \iota \rangle$: Restricting $V^\bullet$ to the component $ \N_{Sp(r)}\subset \N^\iota$ gives a $2$-term complex of $\OO[\langle \iota \rangle]$-modules $V^\bullet|_{\N_{Sp(r)}}$ and the action of $\iota$ gives a $\mathbf{Z}/2\mathbf{Z}$-grading $$V^\bullet|_{\N_{Sp(r)}}\cong V^{\bullet,\iota}\oplus V^{\bullet, mov}$$ into $\iota$-invariant part $V^{\bullet,\iota}$ and moving part $V^{\bullet, mov}$.
Restricting to $\N_{Sp(r)}$, there is a natural map $$\mathbf{L}_\N|^\iota_{\N_{Sp(r)}} \rightarrow \mathbf{L}_{\N_{Sp(r)}}.$$ The proof of the following is analogous to \cite[10.3.2]{Sch}, which works for any involution $\iota$ and a $\iota$-equivariant obstruction theory. This is inspired by the virtual $\C^\times$-localization of \cite[p.4-6]{GP}, see also \cite[4.1.41]{K}. We remark that taking fixed parts is exact (\cite[1.1.18]{Sch}) and that $$V^{\bullet,\iota}=[V^{-1,\iota}\rightarrow V^{0,\iota}]$$ is again a $2$-term complex of vector bundles. Furthermore, this gives us the virtual cotangent bundle for $\N_{Sp(r)}$ according to the following proposition:
\begin{prop}\label{finalthm}
The composition $$V^{\bullet,\iota} \xrightarrow{\psi^\iota} \mathbf{L}_\N|^\iota_{\N_{Sp(r)}} \rightarrow \mathbf{L}_{\N_{Sp(r)}}$$
defines a perfect obstruction theory on $\N_{Sp(r)}$, equipping it with a virtual fundamental class $[\N_{Sp(r)}]^{vir}$.
\end{prop}
\section{Appendix}
\subsection{Dual bundles}
    Let $A,B$ be vector bundles over a variety $S$. Let $A^*=\Homm(A,\OO)$ be the bundle of dual sections. 
    Under the natural isomorphism $A^*\otimes B \xrightarrow{\sim} \Homm(A,B)$, given by $\sigma \otimes b \mapsto \sigma\lrcorner\otimes b  : a \mapsto \sigma(a)b$, taking duals $\Homm(A,B) \rightarrow \Homm(B^*,A^*), \phi\mapsto \phi^*$ is the same thing as swapping the factors $\sigma\otimes b \mapsto b \otimes \sigma$ where a section $b$ of  $B$ acts on a section $\tau $ of $B^*$ via $ev:B\xrightarrow{\sim} B^{**}$. \\
    More precisely, the following diagram 
    \begin{equation*}
        \begin{tikzcd}
        A^*\otimes B \arrow[d,"\mathrm{swap}"] \arrow[r,"\cong"]& \Homm(A,B) \arrow[d,"f\mapsto f^*"] \\
        B\otimes A^* \arrow[r,"\cong"] & \Homm(B^*,A^*)
        \end{tikzcd} 
    \end{equation*}
    commutes.
    \begin{proof}
        Let $a, \tau$ be sections of $A, B^*$ respectively.\\
        As discussed above, following $\sigma \otimes b$ down the LHS gives the map  $$b\lrcorner\otimes \sigma : \tau \mapsto [a \mapsto \tau(b)\sigma (a)].$$
        On the RHS, we get the map $\tau \mapsto (\sigma \lrcorner \otimes b)^* \tau$, given by 
        $$\tau \mapsto [a \mapsto \tau \circ (\sigma\lrcorner \otimes b)(a)=\tau(\sigma(a)b)=\sigma(a)\tau(b)].$$ 
    \end{proof}
        In the case that we have $B=A^*$, swapping the factors $\sigma_1 \otimes \sigma_2\mapsto \sigma_2\otimes\sigma_1$ defines an involution on $\Homm(A,A^*)=A^*\otimes A^*$, splitting into $\pm 1$ eigensheaves $$\curly{S}ym^2(A^*)\oplus \wedge^2A^*.$$ Sections of the former are symmetric (self-dual) maps $\sigma_1\lrcorner\otimes \sigma_2 = \sigma_2\lrcorner\otimes \sigma_1$ and sections of the latter anti-symmetric (skew) $\sigma_1 \lrcorner \otimes \sigma_2=-\sigma_2\lrcorner \otimes \sigma_1$. We remark:
    \begin{rmk}\label{skew}
        A map $\sigma_1 \lrcorner \otimes \sigma_2: A \rightarrow A^*$ is skew if and only if $$\sigma_1(a)\sigma_2(b)=-\sigma_1(b)\sigma_2(a).$$
    \end{rmk}
\subsection{Hilbert polynomials}
    Let $E$ be a coherent sheaf on a smooth projective surface $S$. Let $\OO(1)$ denote a polarisation on $S$ with $H:=c_1(\OO(1))$. Fixing invariants $\ch(E)=(r,c_1(E),c_2(S))$, we set $\deg(E):=\deg(c_1(E).H)$. The Hilbert polynomial $P_E$ of $E$ is given by
    \begin{align}\label{HP}
       P_E(n)= &r\chi(\OO_S)+\frac{1}{2}(c_1(E)c_1(S)+c_1(E)^2-2c_2(E))\\
       \notag &+\frac{rnH}{2}(c_1(S)+2c_1(E))+\frac{r(r+1)}{2}n^2H^2))
    \end{align}

\subsection{Orthosymplectic bundles}\label{orthosymplectic}
    \begin{dfn}
        We call a pair $(E,q)$ consisting of a rank-$r$ vector bundle $E$ over $S$ with a symmetric isomorphism $$q:E \xrightarrow{\sim} E^*$$ an orthogonal bundle.\\
        We call a pair $(E,\omega)$ a symplectic bundle, if we replace $q$ by a skew-symmetric isomorphism $$\omega: E\xrightarrow{\sim} E^*.$$
        We denote the associated pairing in both cases as $e \mapsto e^*=(e_,\_)$.
    \end{dfn}
    \begin{rmk}
        If $E=\Lambda\oplus \Lambda^*$ where $\Lambda \subset E$ is a $\rk$-n sub-bundle, then it admits both a natural symmetric $q$ and skew pairing $\omega$.\\
         Here, the symmetric pairing $$q:\Lambda\oplus \Lambda^* \times \Lambda^* \oplus \Lambda \rightarrow \OO$$ 
        is given by $$((a,\sigma),(\sigma',b)) \mapsto \sigma'(a)+\sigma(b)$$
        and we define $\omega$ by replacing $"+"$ by $"-"$ on the RHS.
    \end{rmk}
    \subsubsection{Orthogonal is skew}
    We remark that we have a splitting $$\Homm(E,E^*)\cong \curly{S}ym^2(E^*)\oplus \wedge^2E^*$$ into symmetric maps $f=f^*$ and skew maps $f=-f^*$.
    \begin{dfn}
        Fixing an orthogonal bundle $(E,q)$, we define the orthogonal group $O(E)\subset GL(E)$ to be those invertible maps $f$ such that $$(f(a),f(b))=(a,b)$$ holds. 
    \end{dfn}
        Accordingly, we define $$\mathfrak{so}(E) \subset \Endd(E)$$ to be the bundle of Lie algebras consisting of all endomorphisms $\phi$ with $$(\phi(a),b)+(a,\phi(b))=0.$$\\
        Writing $\phi=\sigma \lrcorner\otimes e $, this means that $$(\phi(a),b)=\sigma(a) e^*(b)$$ is equal to  $$-(a,\phi(b))=-(\phi(b),a)=-\sigma(b)e^*(a).$$  
    \begin{cor}\label{orthogonalisskew}
        Let $(E,q)$ be an orthogonal bundle. Then the isomorphism $$ q_*: \Endd(E) \xrightarrow{\sim} \Homm(E,E^*),$$
        sending $$\sigma \lrcorner\otimes e \mapsto \sigma\lrcorner \otimes e^*,$$ identifies $$\mathfrak{so}(E) \cong \wedge^2E^*.$$
    \end{cor}
    \begin{proof}
        By the above observation and Rmk. \ref{skew}, $\sigma \lrcorner\otimes e$ is a section of $\mathfrak{so}(E)$ if and only if $\sigma\lrcorner \otimes e^* $ is skew.
    \end{proof}
    \subsubsection{Symplectic is symmetric}
    \begin{dfn}
        Replacing the symmetric structure $q$ by a skew symmetric one $\omega:E\xrightarrow{\sim} E^*$, we have for a symplectic bundle $(E,\omega)$ the sub-sheaf of symplectic maps ${Sp}(E) \subset GL(E)$, i.e. those maps $f$ preserving $\omega$:
        $$(f(a),f(b))=(a,b).$$
        We denote the corresponding bundle of Lie algebras as $$\mathfrak{sp}(E) \subset \Endd(E)$$ satisfying again
        $$(\phi(a),b)+(a,\phi(b))=0.$$
    \end{dfn}
    \begin{cor} \label{symplecticissym}
        The map
       $$\omega_*: \Endd(E) \xrightarrow{\sim} \Homm(E,E^*)$$ identifies $$ \spp(E)  \cong \curly{S}ym^2(E^*).$$ 
    \end{cor}
    \begin{proof}
        An endomorphism $\phi=\sigma\lrcorner \otimes e $ is in $\spp(E)$ if $$(\phi(a),b)=\sigma(a) e^*(b)$$ is equal to  $$-(a,\phi(b))=(\phi(b),a)=\sigma(b)e^*(a),$$ 
        so $\sigma\lrcorner\otimes e $ is in $\mathfrak{sp}(E)$ if and only $\sigma \lrcorner \otimes e^*$ is symmetric. 
    \end{proof}

\end{document}